\documentclass[11pt,a4paper]{article}       
\usepackage[utf8]{inputenc}
\usepackage[english]{babel}
\usepackage{amssymb, amsmath, amsfonts, amsthm}
\usepackage{enumerate}
\usepackage{colonequals}
\usepackage{empheq,fancybox}
\usepackage{esint}
\usepackage[small]{titlesec}

 \newcommand{\hm}[1]{\leavevmode{\marginpar{\tiny%
 $ \hbox to 0mm{\hspace*{-0.5mm} $ \leftarrow $ \hss}%
 \vcenter{\vrule depth 0.1mm height 0.1mm width \the\marginparwidth}%
 \hbox to
 0mm{\hss $ \rightarrow $ \hspace*{-0.5mm}} $ \\\relax\raggedright #1}}}
\newcommand{\euler}{\mathrm{e}} 
\newcommand{\drm}{\mathrm{d}}
\newcommand{\dvol}{\mathrm{dvol}}
\newcommand{\RR}{\mathbb{R}}

\newcommand{\NN}{\mathbb{N}}

\renewcommand{\epsilon}{\varepsilon}
\renewcommand{\phi}{\varphi}

\DeclareMathOperator{\diam}{\mathop{diam}}

\DeclareMathOperator{\Ric}{\mathop{Ric}}
\DeclareMathOperator{\Vol}{\mathop{Vol}}

\newtheorem{thm}{Theorem}[section]
\newtheorem{lem}[thm]{Lemma}
\newtheorem{prop}[thm]{Proposition}
\newtheorem{cor}[thm]{Corollary}
\theoremstyle{definition}
\newtheorem{definition}[thm]{Definition}
\theoremstyle{remark}
\newtheorem{remark}[thm]{Remark}
\usepackage{color}
	\definecolor{darkred}{rgb}{0.5,0,0}
	\definecolor{darkgreen}{rgb}{0,0.5,0}
	\definecolor{darkblue}{rgb}{0,0,0.5}
\usepackage{authblk}

\numberwithin{equation}{section}

\begin{document}
\title{Eigenvalue estimates for Kato-type Ricci curvature conditions}
\author{Christian Rose\footnote{Max Planck Institute for Mathematics in the Sciences, Leipzig, Germany, crose@mis.mpg.de} ~~~~~~~~~~ Guofang Wei\footnote{University of California Santa Barbara, USA, wei@math.ucsb.edu}}
\date{\today}
\maketitle
\begin{abstract}
We prove that optimal lower eigenvalue estimates of Zhong-Yang type as well as a Cheng-type upper bound for the first eigenvalue hold on closed manifolds assuming only a Kato condition on the negative part of the Ricci curvature. This generalizes all earlier results on $L^p$-curvature assumptions. Moreover, we introduce the Kato condition on compact manifolds with boundary with respect to the Neumann Laplacian, leading to Harnack estimates for the Neumann heat kernel and lower bounds for all Neumann eigenvalues, what provides a first insight in handling variable Ricci curvature assumptions in this case.
\end{abstract}
%\tableofcontents
%
%
%
%
%
%
%
%
%
%
%
%
%
\section{Introduction}
A classical field in Riemannian geometry is the investigation of eigenvalue estimates of the Laplace-Beltrami operator depending on geometric assumptions. More precisely, given a compact Riemannian manifold $M=(M^n,g)$ of dimension $n\in\NN$, $n\geq 2$, with possibly non-empty smooth boundary $\partial M$, the aim is to bound the eigenvalues $\lambda$, $\eta$ of the problems
\begin{align}\label{closedeigen}
\Delta u=-\lambda u, \quad \text{if}\quad \partial M=\emptyset,
\end{align}
and 
\begin{align}\label{boundaryeigen}
\Delta u = -\eta u, \quad\partial_\nu u=0\quad \text{if}\quad \partial M\neq \emptyset,
\end{align}
from above and below in terms of curvature restrictions, diameter, and volume. Here, $\partial_\nu$ denotes differentiation w.r.t. the inward normal of $\partial M$, and $-\Delta\geq 0$ the Laplace-Beltrami operator acting in $W^{1,2}(M)$, of course w.r.t. Neumann boundary conditions if $\partial M\neq \emptyset$. Problems \eqref{closedeigen} and \eqref{boundaryeigen} are called the closed and Neumann eigenvalue problem, respectively. Since in both situations the Laplace-Beltrami operator has compact resolvent, the spectra $\sigma_D(-\Delta)$ and $\sigma_N(-\Delta)$ in the cases $\partial M=\emptyset$ and $\partial M\neq\emptyset$, respectively, are purely discrete and consist of unbounded non-decreasing sequences of eigenvalues with finite multiplicties,
\[
\sigma_D(-\Delta)=\{\lambda_k\colon k\in\NN_0\}\quad \text{and}\quad \sigma_N(-\Delta)=\{\eta_k\colon k\in\NN_0\}.
\]
Trivially, $\lambda_0=\eta_0=0$, and both eigenvalues are simple. For $k\geq 1$, it is known that geometric quantities naturally come into play. \\

Problem~\eqref{closedeigen} is well investigated under  pointwise uniform lower Ricci curvature bounds $\Ric\geq K$, $K\in\RR$. The first result was obtained by Lichnerowicz, who proved that if $K>0$, then $\lambda_1(M)\geq nK$ \cite{Lichnerowicz-58}. In the case $\Ric\geq 0$, imposing diameter bounds is necessary, while $\Ric\geq K>0$ automatically implies a diameter bound by Myers' theorem. If $\Ric\geq 0$ and $\diam M\leq D$, 
Zhong and Yang proved a sharp estimate for $\lambda_1$ in \cite{ZhongYang-84},  improving earlier estimate of Li and Yau \cite{LiYau-79}. Many extensions of these estimates can be found e.g. in \cite{Yang-90, Kroger-92,ChenWang-97,HangWang-07,ShiZhang-07,Andrews-Clutterbuck-13, ZhangWang-17}.
Bounds for higher eigenvalues have been proven in \cite{LiYau-86, Gromov-80} for closed manifolds satisfying $\Ric\geq -K$, $K\geq 0$, and $\diam M\leq D$. Upper bounds for $\lambda_1$ under those assumptions have been provided by Cheng \cite{Cheng-75} by comparison principles.\\

In the last decades there was an increasing interest in relaxing the uniform pointwise Ricci curvature assumption to integral Ricci curvature bounds. Those provide estimates that are more stable under perturbations of the metric. Let 
\[
\rho\colon M^n \to\RR, \quad x\mapsto \inf(\sigma(\Ric_x)),
\]
where the Ricci tensor $\Ric$ is viewed as a pointwise endomorphism on the cotangent bundle, and $\sigma(A)$ denotes the spectrum of an operator $A$. For $x\in\RR$, denote $x_-=\max(0,-x)$. The fundamental Laplacian and volume comparison for uniform Ricci curvature lower bounds have been generalized to smallness assumptions on $L^p$- means of $\rho_-$ in \cite{PetersenWei-97} for $p>n/2$. This enabled several authors to extend many results for uniform pointwise Ricci curvature lower bound to the case when the $L^p$-mean of $\rho_-$ is small for some $p>n/2$.  For the first eigenvalue $\lambda_1$, in the seminal work \cite{Gallot-88}, Gallot obtained some lower bound. 
Later, in \cite{PetersenSprouse-98} a Cheng-type comparison estimate for $\lambda_1$ was proven.  Moreover, Aubry showed an optimal Lichnerowicz-type estimate \cite{Aubry-07}. Recently, in \cite{OliveSetoWeiZhang-19} an optimal  Zhong-Yang type estimate was also obtained. \\

In the case of compact manifolds with boundary, i.e., considering Problem~\eqref{boundaryeigen}, the fundamental lower bound is given by Cheeger's inequality \cite{Cheeger-70} as in the closed case. Namely $\eta_1 \ge \frac {h^2}4$, where $h$ is the Neumann isoperimetric constant. In \cite{Buser-82}, Buser obtained a lower bound for $h$, therefore  $\eta_1$, for starlike domain in terms of inner,  outer radius and Ricci curvature lower bound, see also \cite{ChenLi-97}.
In general one needs additional control on the geometry of the boundary.  Li and Yau showed lower bounds for $\eta_k$ if $\partial M$ is convex and Ricci curvature is bounded from below \cite{LiYau-86}. Later, Chen generalized those in \cite{Chen-90} in terms of a uniform lower bound on the Ricci curvature and the second fundamental form plus the so-called interior rolling-$R$-ball condition, see Definition~\ref{def:rolling} (the choice of $R$ depends on the upper bound of
the sectional curvature near the boundary). In \cite{Wang-97}, higher eigenvalue estimates are derived for this set-up. \\

In the $L^p$-Ricci curvature case, nothing is known so far about estimates for $\eta_k$, $k\geq 1$.\\

Another generalization of pointwise Ricci curvature lower bound is the Kato condition for $\rho_-$. %Recall 
A measurable function $V\colon M\to[0,\infty]$ satisfies the Kato condition if there is a $T>0$ such that 
\[
\kappa_T(V):=\sup_{n\in\NN}\Vert \int_0^T P_t (V\wedge n)\drm t\Vert_\infty<1,
\]
where $x\wedge y:=\min(x,y)$, and $(P_t)_{t\geq 0}$ denotes the heat semigroup, i.e.,
\[
P_tf(x):=\euler^{t\Delta}f(x)=\int_M p_t(x,y)f(y)\dvol y,
\]
for $t>0$, where $p_t$ is the heat kernel of $M$. (Note that the truncation procedure ensures that $\kappa_T(V)$ is well-defined, since $P_t$ maps $L^\infty$ to itself.)\\

The Kato condition is a tool from perturbation theory that is used to show semiboundedness of functions $V$ w.r.t. a generator of a Dirichlet form, e.g., $-\Delta$. Moreover, mapping properties of the heat semigroup carry over to the semigroup which is generated by $-\Delta+V$ when $V_-$ satisfies the Kato condition. For an overview we refer to \cite{RoseStollmann-18}. See Section 2 for some background results.

The Kato condition generalizes $L^p$-conditions for $p>n/2$. Namely for closed manifolds the smallness of the $L^p$-mean of $\rho_-$ implies the Kato condition \cite{Batu-14}, see also Lemma~\ref{Lp-Kato}. 
Another advantage of the Kato condition is that for the K\"ahler-Ricci flow, it is known that the $L^4$-norm of the negative part of the Ricci curvature is preserved \cite{TianZhang-16}, but this is not known for $p>4$. On the other hand a version of the Kato condition on $\rho_-$ is preserved \cite{TianZhang-15}.  
Apart from the $L^p$-curvature assumptions, in \cite[Theorem 4.3]{CarronRose-18}, another geometric condition is given which implies the Kato condition. These make the Kato condition a very interesting assumption to study. Note that generalizations of the Laplacian and volume comparison of \cite{PetersenWei-97} for $L^p (p >n/2)$ Ricci curvature are unavailable for Kato condition. \\

 In \cite{RoseStollmann-15}, the authors utilized smallness of $\kappa_T(\rho_-)$ to bound the first Betti number in terms of $\Vert\rho_-\Vert_p$, $p>n/2$. The first author showed in \cite{Rose-16a} a Li-Yau gradient estimate on $(0,T]$ and an estimate on the first Betti number assuming $\kappa_T(\rho_-)$ is small enough via a technique adapted from \cite{ZhangZhu-18}. Carron explored the implications of smallness of $\kappa_T(\rho_-)$ further in \cite{Carron-16}.
Moreover, in \cite{CarronRose-18}, Carron and the first author showed an optimal Lichnerowicz estimate. Compactness results and diameter bounds under a smallness condition on $\kappa_T((\rho-K)_-)$ for $K>0$ have been investigated in \cite{CarronRose-18,Rose-19}. Note that the Kato condition has also been applied on graphs equipped with Bakry-\'Emery Ricci curvature in \cite{MuenchRose-19}.\\

In this paper, we prove all the eigenvalue estimates mentioned above assuming only a Kato-type condition on the negative part of Ricci curvature. First, we generalize the Zhong-Yang type eigenvalue estimate obtained in \cite{OliveSetoWeiZhang-19} to the Kato condition. 

\begin{thm}\label{thm:main1}Let $D>0$ and $n\geq 2$. For any $\alpha>0$ there is a $\delta = \delta (n,D,\alpha) >0$ such that for all closed manifolds $(M^n,g)$ with $\diam(M)\leq D$ satisfying
\[
\kappa_{D^2}(\rho_-)<\delta,
\]
we have
\[
\lambda_1(M)\geq \alpha\frac{\pi^2}{D^2}
\]
with $\alpha\to1$ if $\delta\to 0$.
\end{thm}
Under the same assumptions a lower bound on $\lambda_1$ had been obtained by Carron \cite{Carron-16} (see \eqref{evrough}), but it is not optimal. 
\begin{remark}
It is known that a heat kernel upper bound for all times implies lower estimates for all $\lambda_k$, $k\geq 2$. For the Kato condition, small time estimates for the heat kernel had been obtained first in \cite{Rose-16a} and sharpened and extended to all times in \cite{Carron-16}. Combining the short time bounds, cf. Ineq.~\eqref{heatkernelestimate}, and our sharp quantitative lower bounds for $\lambda_1$, we can improve the constants which would follow from Carron's estimates.

\end{remark}
We also obtain a Cheng-type upper estimate for $\lambda_1$:
\begin{thm}\label{thm:Cheng} For all closed manifolds $(M^n,g)$, there is a constant $C(n)>0$ such that if  $$\kappa_{\diam(M)^2}(\rho_-)\leq \frac 1{16n},$$ then 
\begin{align}
\lambda_1(M)\leq C(n)\diam(M)^{-2}.
\end{align}
\end{thm}
An eigenvalue comparison theorem assuming smallness of the normalized $L^p$-condition, $p>n/2$, for the negative part of the Ricci curvature results from \cite[Theorem 5.1]{PetersenSprouse-98}, yielding a similar bound as in Theorem~\ref{thm:Cheng}. The proof there uses Laplacian comparison for the distance function which is not available for the Kato condition. Theorem~\ref{thm:Cheng} follows from the fact that the Kato condition implies volume doubling. It is well-kown, see, e.g., \cite{HebischSaloffCoste-01}, that the latter property yields the upper estimate on the eigenvalue. For completeness, we will give a proof.

Using the sufficient geometric condition for the Kato condition obtained in  \cite[Theorem 4.3]{CarronRose-18}, we have  the following corollary.

\begin{cor}\label{thm:cor}
Let $D>0$ and $n\geq 2$,   $(M^n,g)$ be a closed manifold with $\diam(M)\leq D$. 
\begin{enumerate}[(i)]
\item 
For any $\alpha\in(0,1)$ there is a $\delta = \delta (n,D,\alpha) >0$ such that if
\[
\int_0^{D}\frac{r}{\Vol(B(x,r))}\int_{B(x,r)}\rho_-\dvol \drm r\leq \delta,
\]
then
\[
\lambda_1\geq \alpha\frac{\pi^2}{D^2}.
\] 

\item
There are $c(n), \epsilon (n)>0$ such that
\[
\int_0^{\diam(M)}\frac{r}{\Vol(B(x,r))}\int_{B(x,r)}\rho_-\dvol \drm r\leq \epsilon (n)
\]
implies
\[
\lambda_1\leq c(n) \diam(M)^{-2}.
\]
\end{enumerate}
\end{cor}

Now we turn to Problem~\eqref{boundaryeigen}. 
To generalize eigenvalue lower bounds on compact manifolds with smooth boundary from the convex to the non-convex case, the following condition was studied by Chen \cite{Chen-90}. 

\begin{definition}\label{def:rolling}
The boundary $\partial M$ of  a Riemannian manifold $M$ satisfies the interior $R$-rolling ball condition if for any $y\in \partial M$ there exists a ball $B(q,R)\Subset M$ such that $\overline{B(q,R)}\cap \partial M=\{y\}$.
\end{definition}

Consider the heat equation for $-\Delta\geq 0$ with Neumann boundary conditions on $\partial M$, i.e.,
\begin{align}\label{eq:Neumannheat}
\partial_t u=\Delta u,\quad\text{in} \ M,\quad \partial_\nu u=0\quad\text{on}\ \partial M.
\end{align}
Similar to the closed case, there is the Neumann heat semigroup $(H_t)_{t\geq 0}$ given by
\[
H_t f(x)=\int_M h_t(x,y)f(y)\dvol y, \quad x\in M, f\in L^2(M), t\geq 0,
\]
where $h_t$ denotes the Neumann heat kernel.
Since $-\Delta\geq 0$ with Neumann boundary condition on $\partial M$ generates a Dirichlet form, we introduce an analogous Kato condition for manifolds with boundary:
\begin{align}\label{KatoNeumann}
\mu_T(\rho_-):=\Vert\int_0^T H_t \rho_-\drm t\Vert_\infty<1.
\end{align}
By proving a Li-Yau gradient estimate for $h_t$ (which clearly also imply upper upper bounds for $h_t$) assuming smallness of $\mu_T(\rho_-)$, we get the following lower bounds for $\eta_1$. Those are the first results for the Neumann boundary problem for the Kato condition. In \cite{Olive-19}, the Neumann condition for $L^p$-Ricci assumptions is studied, but only for compact domains lying in an ambient manifold. In contrast to our results, the estimates obtained there depend highly on the ambient space.

\begin{thm}\label{thm:neumanneigen}
Let $n\in\NN$, $D,H,T,R>0$. Suppose $M$ is a compact Riemannian manifold with boundary of dimension $n$, such that the second fundamental form of $\partial M$ is bounded below by $-H$, and that satisfies the interior $R$-rolling ball condition for $R$ small enough. Assume that $\diam M\leq D$ and $$\mu_T(\rho_-)<\left[2\left((3+2(1+H)^2)(4+8n^2(1+H)^2)-1\right)\right]^{-1}.$$ Then, the first non-zero Neumann eigenvalue satisfies

\[
\eta_1 \geq 
\frac 1{2D^2}2^{-C_2\euler^{16\mu_T(\rho_-)}}\exp\left(-C_1/2D^2-\euler^{16\mu_T(\rho_-)}(1+H)^2D^2\right),
\]
where $C_i=C_i(n,H,R)$, $i\in\{1,2\}$.
\end{thm}

\begin{remark}
As in the closed case, utilizing a Neumann heat kernel upper obtained by a Harnack inequality for the Neumann heat kernel, lower bounds for $\eta_k$, $k\geq 2$ can be derived.
\end{remark}

\begin{remark}\label{rmk:Rsmall}
As in \cite{Chen-90, Wang-97} the assumption that the interior $R$-rolling ball condition holds for $R>0$ small enough depends implicitly on the upper bound $K_R$ of the sectional curvature of the $R$-tubular neighborhood of $\partial M$ in $M$. More precisely, $R\in(0,1)$ has to be chosen such that 
\[
\sqrt{K_R}\tan\left(R\sqrt{K_R}\right)\leq\frac 12(1+H)\quad\text{and}\quad \frac{H}{\sqrt{K_R}}\tan\left(R\sqrt{K_R}\right)\leq\frac 12.
\]
\end{remark}

The plan of this paper is as follows: in Section~\ref{section:closed} we focus on the closed eigenvalue Problem \eqref{closedeigen}. We review known results from \cite{Rose-16a, Carron-16} that will be used to derive Theorems~\ref{thm:main1}. 
Then, we adapt techniques from \cite{RoseStollmann-15} to show certain estimates for perturbed semigroups generated by the Schr\"odinger operator $-\Delta-\rho_-$ assuming smallness of $\kappa_T(\rho_-)$. We utilize them to show estimates for the smallest eigenfunction and the associated eigenvalue of this operator. Those estimates also hold for general continuous potentials $V$ instead of $\rho_-$. Then, we adapt the proof of \cite[Theorem~1.1]{OliveSetoWeiZhang-19} to derive Theorem~\ref{thm:main1}. 
After that, we prove upper bounds on the first Dirchlet eigenvalue in balls to derive Theorem~\ref{thm:Cheng}. Although we derive an estimate of $\lambda_1(M)$ similar to Cheng's explicit bounds, we are neither able to prove a sharp bound nor sharp comparison results for Dirichlet eigenvalues in balls. This might be a good direction for further investigations.\\
In Section~\ref{section:boundary}, we focus on Problem~\eqref{boundaryeigen}, i.e., manifolds with non-empty boundary $\partial M\neq \emptyset$. Transferring the technique from \cite{Wang-97,Olive-19} to our setting, we derive a Harnack inequality for positive solutions of the heat equation with Neumann boundary conditions under the interior $R$-rolling ball condition and lower-bounded second fundamental form. In contrast to the main results in \cite{Olive-19}, smallness of $\mu_T(\rho_-)$ surprisingly allows to get rid of all extra assumptions made in the latter article to get an a priori Neumann heat kernel bound. The proof is almost the same and we explain the necessary changes. To obtain a lower bound for $\eta_1$, we adapt the technique from \cite{ChenLi-97} via a much shorter proof than in the latter article. This is possible by utilizing the global Harnack inequality obtained by us, which was not known in their setting. Finally, we use the Harnack inequality to provide upper bounds for the Neumann heat kernel $h_t$. \\

We want to emphasize that the Kato condition on the Neumann heat semigroup could be generalized to Robin boundary conditions. We conjecture that similar estimates as for the Neumann eigenvalues should hold under sufficiently regular assumptions on the boundary functions.\\

\textit{Acknowledgement:} C.~R. wants to thank G.~W. and UCSB for their hospitality during his stay and for providing a very nice environment and Florentin M\"unch and Xavier Ramos Oliv\'e for useful discussions. G.~W. is partially supported by NSF Grant DMS 1811558.

\section{Eigenvalue bounds for closed manifolds}\label{section:closed}
\subsection{Preliminaries}
The spectrum of a selfadjoint operator $A$ in a Hilbert space $X$ will always be denoted by $\sigma(A)$. $\Vert A\Vert_{p,q}$ will be the operator norm of a bounded operator $A\colon L^p(M)\to L^q(M)$.
For a closed Riemannian manifold $M$ let 
$\dvol$ its volume measure, and for $x\in M$ and $r>0$, let $B(x,r)$ the ball with center $x$ and radius $r$. It is convenient to abbeviate 
$$V(x,r):=\Vol(B(x,r)).$$

Denote by $((-\Delta+\alpha)^{-1})_{\alpha>0}$ the resolvents of $\Delta$. 
The heat semigroup and the resolvents can be extended to $L^p(M)$ for $p\in [1,\infty]$. \\

Note that the functions $\rho$ and $\rho_-$ introduced above are continuous, such that we can restrict our considerations involving the Kato constant to continuous functions. Thus, for some continuous function $V\colon M\to[0,\infty]$, $\alpha,T>0$, we simplify
\begin{align}\label{Kato1}
\kappa_T(V):=\Vert \int_0^T P_tV\drm t\Vert_\infty
\end{align}
and 
\begin{align}\label{Kato2}
c_\alpha(V):=\Vert (-\Delta+\alpha)^{-1}V\Vert_\infty.
\end{align}

Note that those two quantities are closely related \cite{RoseStollmann-15,Batu-14}: 
\begin{align}\label{connection}
\left(1-\euler^{-\alpha\beta}\right)c_\alpha(V)\leq \kappa_\beta(V)\leq \euler^{\alpha\beta}c_\alpha(V).
\end{align}
Sometimes it is more convenient to work with $\kappa_T$ instead of $c_\alpha$ and vice versa. We say that $V\geq 0$ measurable satisfies the Kato condition if 
\[
\kappa_T(V)<1 \quad \text{or}\quad c_\alpha(V)<1
\]
for some $T>0$ or $\alpha>0$ respectively. Of particular interest for us is the Kato condition for $\rho_-$:
\[
\kappa_T(\rho_-)<1 \quad \text{or}\quad c_\alpha(\rho_-)<1
\]
or variants of those with functions depending on $\rho$.\\

Denote
\[
\Vert f\Vert^*_p=\left(\frac 1{\Vol(M)}\int_M \vert f\vert^p \ \dvol\right)^{\frac 1p}.
\] 

The next small lemma shall illustrate how smallness of $\Vert\rho_-\Vert_p^*$ implies smallness of the Kato constant. It can be found, e.g., in \cite{RoseStollmann-15}, and is a special case of \cite{Batu-14}.
\begin{lem}  \label{Lp-Kato} Let $n\geq 3$, $p>n/2$, $D>0$. There exist $\epsilon=\epsilon(n,p,D)>0$ and $C=C(n,p,D)>0$ such that for all manifolds $M=(M^n,g)$ satisfying $\Vert\rho_-\Vert_p^*\leq \epsilon$, we have $$\kappa_{D^2}(\rho_-)\leq C\Vert\rho_-\Vert_p^*.$$
\end{lem}
\begin{proof}
The assumption $\Vert \rho_-\Vert_p^*\leq\epsilon$ implies a heat kernel estimate \cite{Gallot-88,PetersenSprouse-98} of the form 
\[
p_t(x,y)\leq \frac{C(n,p,D)}{\Vol(M)}t^{-n/2},\quad x,y\in M, t\in(0,D^2]. 
\] 
Moreover, we have by Dunford-Pettis' theorem
\[
\Vert P_t\Vert_{1,\infty}=\sup_{x,y\in M} p_t(x,y),
\]
provided the right-hand side is bounded.
The Riesz-Thorin interpolation theorem implies
\[
\Vert P_t\Vert_{p,\infty}=\sup_{x,y\in M} p_t(x,y)^{1/p}.
\]
Thus, considering $\rho_-$ as an operator from $L^\infty(M)$ to $L^p(M)$ for $p>n/2$ and factorizing through $L^p$ via the Grothendieck theorem, we get 
\[
\Vert P_t\rho_-\Vert_{\infty}=\Vert P_t\rho_-\Vert_{\infty,\infty}\leq \Vert P_t\Vert_{p,\infty}\Vert \rho_-\Vert_{\infty,p}=\Vert P_t\Vert_{p,\infty}\Vert \rho_-\Vert_{p}.
\]
Thus, for $p>n/2$, 
\[
\kappa_{D^2}(\rho_-)\leq \Vert\rho_-\Vert_p\int_0^{D^2} \Vert P_t\Vert_{p,\infty}\drm t\leq C\Vert\rho_-\Vert_p^*\int_0^{D^2} t^{-n/2p}\drm t. 
\]
\end{proof}

We recall the following estimates for the heat kernel, the volume doubling condition, and the smallest positive eigenvalue of the Laplacian proved in \cite{Carron-16,Rose-16a}:
Let $\nu:=e^2n$ and $T$ be the largest possible time such that $\kappa_T(V)\leq \frac 1{16n}$.
\begin{thm}[\cite{Carron-16,Rose-16a}]\label{harnackkato}Let $T>0$, $n\geq 2$. For all $n$-dimensional compact $M$ satisfying $$\kappa_{T}(\rho_-)\leq \frac 1{16n},$$ we have for any positive solution $u$ satisfying 
$$\partial_t u=\Delta u$$
and all $t\in(0,T]$
\begin{align}
\euler^{-2}\frac{\vert\nabla u\vert^2}{u^2}-\frac{\partial_t u}{u}\leq \frac{\nu}{2t}.
\end{align}
In particular, we have the Harnack inequality for such $u$: for all $x,y\in M$, $0<s\leq t\leq T$,
\begin{align}\label{thm:harnackclosed}
u(s,x)\leq \left(\frac ts\right)^{\frac \nu 2}\euler^{2\frac{d(x,y)^2}{t-s}}u(t,y).
\end{align}
\end{thm}

Note that the statement above indeed holds without any restrictions on the diameter. In the articles cited in theorem, the authors also obtained heat kernel estimates. Carron explored in \cite{Carron-16} the implications of the Harnack inequality further, and obtained also the volume doubling property as well as an eigenvalue and a heat kernel lower bound. By choosing $T=D^2$, the result below follows immediately from this article.
\begin{thm}[\cite{Carron-16}]\label{thm:carron} Let $D>0$, $n\geq 2$ For all $M$ satisfying $\diam(M)\leq D$ and $$\kappa_{D^2}(\rho_-)\leq \frac 1{16n},$$ then 
\begin{enumerate}[(i)]
\item there is a doubling constant $c_d(n, D)>0$ such that for all $x\in M$, $0<r\leq R\leq D$, 
\begin{align}
\frac{V(x,R)}{V(x,r)}\leq c_d\left(\frac Rr\right)^\nu.
\end{align}
\item there is a constant $c(n)>0$ such that for all $t\in(0,D^2/2]$, $x,y\in M$,
\begin{align}\label{heatkernelestimate}
p_t(x,y)\leq \frac{c(n)D^{\nu}}{\Vol(M)}t^{-\nu/2}.
\end{align}
\item there is a constant $c=c(n)>0$ such that
\begin{align}\label{evrough} 
\lambda_1(M)\geq \Lambda:=\frac{c}{D^2}.
\end{align}
\item there exists $\epsilon(n)>0$ such that for all $t\in(0,D^2/2]$
\begin{align}\label{lowerhk}
p_t(x,y)\geq \frac{\epsilon(n)}{V(x,\sqrt t)}.
\end{align}
\end{enumerate}
\end{thm}

In the following, we let $$\bar w=\frac 1{\Vol(M)}\int w\ \dvol$$.

Note that \eqref{evrough} implies a Poincar\'e inequality:
\begin{equation}  
\Lambda\leq \lambda_1(M)=\inf_{f\perp 1}\frac{\int \vert\nabla f\vert^2\dvol}{\int f^2\dvol}=\inf_{w \neq constant} \frac{\int \vert\nabla w\vert^2\dvol}{\int \vert w-\bar w\vert^2\dvol},
\end{equation}
such that for all $w\in W^{1,2}(M)$, we have
\begin{equation}\label{poincare-ineq}
\Vert w-\bar w\Vert_2^2\leq \Lambda^{-1}\Vert\nabla w\Vert_2^2.
\end{equation}
\subsection{Proofs of Theorems~\ref{thm:main1} and \ref{thm:Cheng}}

The proof of Theorem~\ref{thm:main1} is adapted from \cite{OliveSetoWeiZhang-19}. In \cite{OliveSetoWeiZhang-19} the key estimate is to control the auxilary function $J$ which is  the solution of the following problem, 
\begin{align}
\Delta J-\tau\frac{\vert \nabla J\vert^2}{J}-2J\rho_-=-\sigma J  \label{J-eqn}
\end{align}
for $\tau>1$ and $\sigma\geq 0$. 
Using $J=w^{-\frac{1}{\tau -1}}$, this equation is equivalent to
\begin{align}\label{eigeneq}
\Delta w+Vw=\tilde\sigma w,
\end{align}
where $V:=2(\tau-1)\rho_-$ and $\tilde\sigma:=(\tau-1)\sigma$. The proof of those estimates is based on bounds of the Sobolev constant and smallness of $\Vert\rho_-\Vert_p^*$ using a Moser-iteration. This does not carry over to our setting, and we use semigroup techniques instead. 

We will study the lowest eigenfunction of the general Schr\"odinger operator $-\Delta+V$ for $V\colon M\to [0,\infty)$ being continuous and satisfing the Kato condition since this might be of independent interest.  We show the necessary comparison estimate for the eigenfunction $w$ and the function $J$ depending on the Kato condition. Then we apply to $V =2(\tau-1)\rho_-$ and proving  Proposition~\ref{prop:eigen} and \ref{prop:aux}, replacing  Propositions~2.2 and 2.3 in \cite{OliveSetoWeiZhang-19},

From the heat kernel upper bound \eqref{heatkernelestimate}, we infer the following mapping properties for perturbed heat semigroups.

\begin{prop}\label{sgestimate}
Let $T,D>0$, $n\geq 2$. For all $M^n$ with $\diam(M)\leq D$ and $\kappa_{D^2}(\rho_-)<\frac 1{16n}$,  and all continuous functions $V\colon M\to[0,\infty)$ satisfying
$$\kappa = \kappa_T(V)<1,$$
we have for all $t\in(0,D^2/2]$,
\begin{align}
\Vert \euler^{-t(-\Delta-V)}\Vert_{2,\infty}\leq \left(\frac{1}{1-\kappa}\right)^{\frac12\left(1+\frac tT\right)}\left(\left(\frac{2}{1-\kappa}\right)^{\frac{1+\kappa}{1-\kappa}\left(1+\frac tT\right)+\nu/2} \frac{c_nD^{\nu/2}}{\Vol(M)} t^{-\nu/2}\right)^{\frac 12}.
\end{align}
\end{prop}
\begin{proof}
By \eqref{heatkernelestimate} for all $t\in (0,D^2/2]$ 
\begin{align}
\Vert \euler^{t\Delta}\Vert_{1,\infty}=\sup_{x,y\in M}p_t(x,y)\leq \frac{c(n)D^{\nu/2}}{\Vol(M)}t^{-\nu/2}. 
\end{align}
Moreover, we have for all $t>0$ as in \cite{RoseStollmann-15} and \cite{Voigt-86}
\begin{align}
\Vert\euler^{t(\Delta+V)}\Vert_{1,1}\leq \left(\frac{1}{1-\kappa}\right)^{1+\frac tT}.
\end{align}
Following the proof of \cite[Corollary~5.4]{RoseStollmann-15}, we get for all $t\in(0,D^2/2)]$,
\begin{align}
\Vert\euler^{t(\Delta+V)}\Vert_{2,\infty}\leq \left(\frac{1}{1-\kappa}\right)^{\frac12\left(1+\frac tT\right)}\left(\left(\frac{2}{1-\kappa}\right)^{\frac{1+\kappa}{1-\kappa}\left(1+\frac tT\right)+\nu/2} \frac{c_nD^{\nu/2}}{\Vol(M)} t^{-\nu/2}\right)^{\frac 12}.
\end{align}
\end{proof}

 Let $\tilde\sigma$  be the the smallest eigenvalue of $-\Delta-V$:
\begin{align}\label{eq:eigen}
\Delta w+Vw=\tilde\sigma w.
\end{align}
We give an upper and lower bound on $\tilde\sigma$ in terms of $c_\alpha(V)<1$ (see \eqref{Kato2} for the definition of $c_\alpha(V)$). 

\begin{prop}\label{prop:eigen} Assume that $V\geq 0$ and $c_\alpha(V)<1$ for some $\alpha>0$.
 then we have
\[
0\leq \tilde\sigma\leq \alpha c_\alpha.
\]
\end{prop}

\begin{proof}
Since $-\tilde\sigma$ is the smallest eigenvalue of $\Delta+V$, we can choose $w\geq 0$ such that $\Vert w\Vert_2^*=1$. Integrating \eqref{eq:eigen} over $M$, we get
\[
\int Vw \ \dvol = \tilde\sigma \int w\ \dvol.
\]
Since $V$ and $w$ are non-negative, $\tilde\sigma\geq 0$.

For the upper bound,  recall general perturbation theory implies \cite{RoseStollmann-18, StollmannVoigt-96,Simon-82}
\begin{equation}\label{relative}
\int V\phi^2 \leq c_\alpha \int \vert \nabla \phi\vert^2 +\alpha c_\alpha \int \phi^2, \quad \phi\in W^{1,2}(M).
\end{equation}
Now multiply \eqref{eq:eigen} by $w$, integrate over $M$, and use \eqref{relative} to obtain
\begin{align*}
\Vol(M)\tilde\sigma&=\tilde\sigma\int w^2\dvol=-\int \vert \nabla w\vert^2\dvol + \int Vw^2\dvol\\
&\leq -\int\vert \nabla w\vert^2+c_\alpha\int\vert\nabla w\vert^2 \ \dvol+\alpha c_\alpha\int w^2 \ \dvol\\
&\leq \alpha c_\alpha \Vol(M),
\end{align*}
due to $c_\alpha(V)<1$. Thus,
\[
\tilde\sigma\leq \alpha c_\alpha.
\]
\end{proof}
Now we estimate the first eigenfunction $w$.
\begin{lem}
Let $w$ be a first eigenfunction such that $\Vert w\Vert_2^*=1$ and assume $c_\alpha(V)<1$. Then,
\begin{align}\label{nicebound}
\int \vert \nabla w\vert^2\leq \frac{\alpha c_\alpha}{1-c_\alpha}\Vol(M)
\end{align}
and
\begin{align}\label{twonorm}
\Vert w-\bar w\Vert_2^2\leq \Lambda^{-1}\frac{\alpha c_\alpha}{1-c_\alpha}\Vol(M).
\end{align}
\end{lem}
\begin{proof}
From \eqref{eq:eigen} we infer
\begin{align*}
\int\vert \nabla w\vert^2 \dvol &=-\int w\Delta w\dvol = \int w(V-\tilde\sigma)w\dvol=\int Vw^2\dvol  -\tilde\sigma \int w^2\dvol\\
&\leq c_\alpha\int \vert\nabla w\vert^2\dvol+(\alpha c_\alpha-\tilde\sigma)\int w^2\dvol,
\end{align*}
where we use \eqref{relative} in the last inequality.
Thus,
\begin{align}
\int \vert \nabla w\vert^2\leq \frac{\alpha c_\alpha-\tilde\sigma}{1-c_\alpha}\Vol(M).
\end{align}
Inequality \eqref{poincare-ineq} implies 
\begin{align}
\Vert w-\bar w\Vert_2^2\leq \Lambda^{-1}\Vert\nabla w\Vert^2_2\leq \Lambda^{-1} \frac{\alpha c_\alpha-\tilde\sigma}{1-c_\alpha}\Vol(M).
\end{align}
\end{proof}

\begin{lem} Suppose $c_\alpha(V)<1$, $\kappa_T(V)<\frac12$ and $\kappa_{D^2}(\rho_-)\leq \frac 1{16n}$.
We have 
\begin{align}\label{inftynorm}\Vert w-\bar w\Vert_\infty\leq 2^{\frac{3D^2}{2T}}c_{n,\nu} D^{-\nu/2}\frac{1-\kappa_T(V)}{1-2\kappa_T(V)}\Lambda^{-1/2} \sqrt{\frac{\alpha c_\alpha(V)}{1-c_\alpha(V)}}.
\end{align}
\end{lem}
\begin{proof}
Define $h:= w-\bar w$ and
consider the heat equation
\begin{align}\label{heateq}
\begin{cases}
\Delta f+Vf=\partial_t f & \text{ on } M\times (0,T],\\
f= \vert h\vert & \text{ on } M\times \{0\}.
\end{cases}
\end{align}
The general solution of this problem is given by $f=\euler^{-t(-\Delta-V)}\vert h\vert$. To derive an upper bound on $\Vert h\Vert_\infty$, we transform \eqref{heateq} into an integral equation via Duhamel's formula:
\begin{align}\label{integraleq}
f(x,t)= \vert h\vert(x) +\int_0^t\int_M p_{t-s}(x,y)V(y)f(x,s)\dvol(y)\drm s,
\end{align}
where $p_t$ denotes the heat kernel. Define $g(t):=\max_{x\in M, s\in [0,t]}f(x,s)$. Note that $g$ is non-decreasing in $t$. Thus, for $t\in(0,T]$,
\begin{align*}
g(t)&\leq \Vert h\Vert_\infty +\Vert\int_0^t \int_M p_{t-s}(x,y)V(y)g(s)\dvol(y)\drm s\Vert_\infty\\
&\leq \Vert h\Vert_\infty +g(t)\Vert\int_0^t\int_M p_{t-s}(x,y)V(y)\dvol(y)\drm s\Vert_\infty\\
&= \Vert h\Vert_\infty +g(t)\Vert\int_0^t\int_M p_{r}(x,y)V(y)\dvol(y)\drm r\Vert_\infty\\
&\le \Vert h\Vert_\infty +g(t) \kappa_T(V).
\end{align*}
Hence,
\[
g(t)\leq \frac{\Vert h\Vert_\infty}{1-\kappa_T(V)}.
\]
Thus, we infer from \eqref{integraleq}
\begin{align*}
\vert h\vert(x)&\leq \vert f(x,t)\vert +g(t)\Vert\int_0^t \int_M p_{t-s}(x,y)V(y)\dvol(y)\drm s\Vert_\infty\\
&\leq \vert f(x,t)\vert+ \frac{\Vert h\Vert_\infty}{1-\kappa_T(V)}\kappa_T(V)\\
&= \vert \euler^{-t(-\Delta-V)}\vert h\vert(x)\vert+\frac{\Vert h\Vert_\infty}{1-\kappa_T(V)}\kappa_T(V)\\
&\leq \Vert \euler^{-t(-\Delta-V)}\vert h\vert\Vert_\infty+\frac{\Vert h\Vert_\infty}{1-\kappa_T(V)}\kappa_T(V)\\
&\leq \Vert \euler^{-t(-\Delta-V)}\Vert_{2,\infty}\Vert h\Vert_2+\frac{\Vert h\Vert_\infty}{1-\kappa_T(V)}\kappa_T(V).
\end{align*}
Hence,
\begin{align}
\Vert h\Vert_\infty\leq \frac{1-\kappa_T(V)}{1-2\kappa_T(V)}\Vert\euler^{-t(\Delta-V)}\Vert_{2,\infty}\Vert h\Vert_2,
\end{align}
which is finite if $\kappa_T(V)<\frac 12$. From Proposition~\ref{sgestimate}, we get with $t=D^2/2$
\begin{align}
\Vert \euler^{t(\Delta+V)}\Vert_{2,\infty}\leq 2^{\frac{3D^2}{2T}}c_{n,\nu} \frac{D^{-\nu/2}}{\Vol(M)^{\frac 12}},
\end{align}
where $c_{n,\nu}= 8^{\frac 32+\frac\nu4}c(n)^{\frac12}$. Thus, 
\begin{align*}
\Vert h\Vert_\infty&\leq 2^{\frac{3D^2}{2T}}c_{n,\nu} \frac{D^{-\nu/2}}{\Vol(M)^{\frac 12}} \frac{1-\kappa_T(V)}{1-2\kappa_T(V)}\Vert h\Vert_2\\
&\leq2^{\frac{3D^2}{2T}}c_{n,\nu} \frac{D^{-\nu/2}}{\Vol(M)^{\frac 12}}\frac{1-\kappa_T(V)}{1-2\kappa_T(V)}\Lambda^{-1/2} \sqrt{\frac{\alpha c_\alpha(V)}{1-c_\alpha(V)}\Vol(M)}.
\end{align*}
\end{proof}

\begin{lem} In the situation above, if 
\[
2^{\frac{3D^2}{2T}}c_{n,\nu} D^{-\nu/2}\frac{1-\kappa_T(V)}{1-2\kappa_T(V)}\Lambda^{-1/2} \sqrt{\frac{\alpha c_\alpha(V)}{1-c_\alpha(V)}}<\frac 12,
\]
we have $\bar w>\frac 12$.
\end{lem}
\begin{proof}
The lemma above tells us that 
\[
w\leq \bar w+K,
\]
where we set 
\[
K:=2^{\frac{3D^2}{2T}}c_{n,\nu} D^{-\nu/2}\frac{1-\kappa_T(V)}{1-2\kappa_T(V)}\Lambda^{-1/2} \sqrt{\frac{2\alpha c_\alpha}{1-c_\alpha}}.
\]
Since $\Vert w\Vert_2^*=1$ and $\bar w\leq 1$, we have 
\[
\Vol(M)=\int w^2\dvol\leq \int (\bar w+K)^2\dvol= \Vol(M)\left(\bar w^2+2K+K^2\right).
\]
Hence, if $K<\frac 12$, we get 
\[
\frac 14<1-K-K^2\leq \bar w^2.
\]
\end{proof}

We are now in the position to prove estimates on the auxiliary function $J$ defined in \eqref{J-eqn}, which is the key to derive Theorem~\ref{thm:main1} from the proof of \cite[Theorem~1.1]{OliveSetoWeiZhang-19}.
\begin{prop}\label{prop:aux}Suppose $D>0$, $n\geq 2$. For any $\delta\in[0,1)$ there exists a $\tau_0>1$ such that for all $\tau\geq\tau_0$ and all manifolds satisfying $\diam(M)\leq D$ and
\begin{align}\label{condition1}
\kappa_{D^2}(\rho_-)< \frac 12 (\tau-1)^{-3},
\end{align}
we have 
\[
\vert J-1\vert\leq \delta.
\]
\end{prop}
\begin{proof}
The choice $\tau\geq \sqrt{8n}+1$ implies $k_{D^2}(\rho_-)\leq \frac 1{16n}$. Since $V=2(\tau-1)\rho_-$  the condition~\eqref{condition1} implies $\kappa_{D^2}(V)<\frac 12$. Note that $(\tau-1)^2-2\geq 7$, such that \eqref{connection} gives for all $\alpha\geq D^{-2}\ln\left(\frac 12(\tau-1)^2\right)$
\[
c_\alpha(V)\leq 2(\tau-1)\kappa_{D^2}(\rho_-)\left(1-\euler^{-\alpha D^2}\right)^{-1}\leq (\tau-1)^{-2}\left(1-\euler^{-\alpha D^2}\right)^{-1}< \frac 12.
\]
Hence,
\begin{align}\label{inftynorm2}
\Vert w-\bar w\Vert_\infty\leq C_{D,\Lambda}\sqrt{\alpha c_\alpha(V)}
\end{align}
with $C_{D,\Lambda}:=4c_{n,\nu}\frac{\sqrt{8n}}{\sqrt{8n}-2}D^{-\nu/2}\Lambda^{-1/2}$.
If we choose $\alpha:=D^{-2}(\tau-1)$, then $1-\euler^{-\alpha D^2}\leq \frac 12$ and
\begin{align}
\Vert w-\bar w\Vert_\infty\leq \sqrt 2C_{D,\Lambda}D^{-1}(\tau-1)^{-1/2}.
\end{align}
To get $\bar w>\frac 12$, choose $\tau$ such that 
\begin{align}
\sqrt 2C_{D,\Lambda}D^{-1}(\tau-1)^{-1/2}<\frac 12.
\end{align}
Thus, it suffices to take 
\[
\tau\geq\tau_1:=\max\left(\sqrt{8n}+1,2\left(1+4\frac{C_{D,\Lambda}^2}{D^2}\right)\right)
\]
such that the above estimates are valid.

With those choices, we can finish the proof of the proposition as in \cite{OliveSetoWeiZhang-19}. Consider the function $w_2:= w\bar w^{-1}$. Then, by \eqref{inftynorm} above,
\[
1-\tilde\delta\leq w_2\leq 1+\tilde\delta,
\]
where $\tilde\delta:=2 \sqrt 2C_{D,\Lambda}D^{-1}(\tau-1)^{-1/2}$. We define $J:=w_2^{-\frac 1{\tau-1}}$. Use the first order Taylor expansion of $f(x)=x^{-\frac 1{\tau-1}}$ near one on the domain $(1-\tilde\delta,1+\tilde\delta)$, i.e., $f(x)=1+R_1(x)$ with 
\[
\vert R_1\vert\leq \vert f'(x^*)(x-1)\vert\leq \frac{2\delta}{(\tau-1)(1-\delta)^{\frac\tau{\tau-1}}},
\]
where $x^*\in (1-\tilde\delta,1+\tilde\delta)$. From above, we know that we have to choose $\tau\geq\tau_1$. Moreover, $\tau$ must be so large such that
\[
\frac{2\tilde\delta}{(\tau-1)(1-\tilde\delta)^{\frac\tau{\tau-1}}}\leq\delta
\]
is satisfied. 
Hence, we can choose
\[
\tau_2:=1+\max\left(8\frac{C_{D,\Lambda}^2}{D^{2}},\left(\frac{162C_{D,\Lambda}^2}{D^2\delta^2}\right)^\frac{1}{3}\right)
\]
and $\tau_0:=\max\{\tau_1,\tau_2\}$.
\end{proof}

%
%
%
%
%
%
%
%
%\subsection{Proof of Theorem~\ref{thm:Cheng}}
%
%
%
%
%
%
%
%
%
Now, we turn to the proof of Theorem~\ref{thm:Cheng}. The statement follows from a simple test-function argument using the volume doubling condition and Courant's principle on the first eigenvalue.
\begin{lem}\label{lem:dirichletball} Assume 
\[
\kappa_{\diam M^2}(\rho_-)\leq \frac 1{16n}.
\]
Then there is a $D=D(n,\nu)>0$ such that for any $R< \diam M$, we have 
\[
\lambda_1(B(x,R))\leq \frac{D}{R^2}.
\]
\end{lem}
\begin{proof} Fix $x_0\in M$ and $R<\diam M$. 
Define $f(x):= \varphi(d(x,x_0))$ with 
\[
\varphi(r)=\begin{cases} 1&\colon r\in[0,R/2]\\ -\frac 2R r+2&\quad r\in(R/2,R].\end{cases}
\]
Then $f\in W_0^1(M)$ with $\nabla f(x)= \varphi'(d(x_0,x))\nabla d(x_0,x)$ a.e. and $\vert\nabla f\vert^2=\frac{4}{R^2}\mathbf{1}_{B(x_0,R)\setminus B(x_0,R/2)}$. Thus,
\[
\Vert \nabla f\Vert^2_{B(x_0,R)}=\frac 4{R^2}V(x_0,R)\quad\text{and}\quad \Vert f\Vert_{2,B(x_0,R)}^2\geq V(x_0,R/2).
\]
Hence, using $f$ as test function for the Rayleigh quotient of $\lambda_1(B(x_0,R))$ and the volume doubling property, see Theorem~\ref{thm:carron}(i),
\[
\lambda_1(B(x_0,R))\leq \frac{\Vert\nabla f\Vert_{2,B(x_0,R)}^2}{\Vert f\Vert_{2,B(x_0,R)}^2}\leq \frac{4}{R^2}\frac{V(x_0,R)}{V(x_0,R/2)}\leq \frac{4c_d2^\nu}{R^2}.
\]
\end{proof}
%Assume $\kappa_{\diam M^2}(\rho_-)\leq \frac 1{16n}$. According to Carron, volume doubling is satisfied for all balls with doubling constant $C_n$. Hence, we get

\begin{proof}[Proof of Theorem~\ref{thm:Cheng}]
Lemma~\ref{lem:dirichletball} tells us that for all balls of radius $R<\diam M$, we have 
\[
\lambda_1(B(x,R))\leq \frac{D}{R^2}.
\] 
Take $R=\diam M/2$ and note that there are points $x_0,x_1\in M$ such that $B(x_0,R)\cap B(x_1,R)=\emptyset$. By Courant's principle,
\[
\lambda_1(M)\leq \max(\lambda_1(B(x_0,R)),\lambda_1(B(x_1,R)))\leq \frac{D}{R^2}.
\]
\end{proof}

\section{Eigenvalue bounds for compact manifolds with boundary}\label{section:boundary}
\subsection{Gradient estimates and Harnack inequality for the Neumann heat kernel}

We emphasized in the introduction that to obtain a lower bound on $\eta_1$, we will prove a Li-Yau gradient estimate for positive solutions of the Neumann heat equation \eqref{eq:Neumannheat} under the Neumann Kato condition and suitable regularity of the boundary. Then, we adapt the technique from \cite{ChenLi-97} to our setting, giving a much shorter proof than in the latter article.
The proof of the Li-Yau inequality is an adaption from \cite{Olive-19}. In this article, the author proved a Li-Yau gradient estimate for positive solutions of the Neumann heat equation in relatively compact domains $M\subset N$ in a globally doubling Riemannian manifold $N$ assuming the interior $R$-rolling ball condition in $\partial M$, second fundamental form bounded below, and uniform smallness of $\Vert\rho_-\Vert_p^*$ in $M$. Those conditions are made to ensure an a priori upper bound on the Neumann heat kernel, which is used to prove the Li-Yau gradient estimate. In contrast, using the Kato condition for the Neumann heat semigroup, we neither need the existence of an ambient manifold $N$ nor the doubling condition to derive a Li-Yau gradient estimate.
In partiular, we do not need an a priori Neumann heat kernel bound to derive the Li-Yau inequality, and the heat kernel bound follows a posteriori.

\begin{thm}\label{gradientestimate}
Let $n\in\NN$, $n\geq 2$, $H,T,R>0$. Suppose $M$ is a compact Riemannian manifold with boundary of dimension $n$, such that the second fundamental form of $\partial M$ is bounded below by $-H$, and that satisfies the interior $R$-rolling ball condition for $R$ small enough (see Remark~\ref{rmk:Rsmall}). Assume $$\mu_T(\rho_-)<\left[2\left((3+2(1+H)^2)(4+8n^2(1+H)^2)-1\right)\right]^{-1}.$$ Then, we have for any positive solution of the Neumann heat equation on $M$:
\begin{align}
\frac{1}{2(1+H)^2} J \frac{\vert \nabla u\vert^2}{u^2}-\frac {\partial_t u}u\leq C_1+\frac{C_2}{J t},\quad t\in(0,T],
\end{align}
where 
\[
\euler^{-16\mu_T(\rho_-)}\leq J\leq 1
\]
and 
\[
C_1=4n^2(1+H)^2\sqrt{\frac {4(1+2n^2(1+H)^2)}{4(1+2n^2(1+H)^2)-1}}c_1,
\]
\[
c_1=\frac{128n^2H^2}{R^2}+\frac{H}{2R(1+H)(1+2n^2(1+H)^2}+\frac{16H(1+H)}{R},
\]
\[
C_2=\frac 1{2n^2(1+H^2)}\left(1-\frac{1}{4+8n^2(1+H)^2}\right).
\]
\end{thm}

To make this statement comparable it with the result in \cite{LiYau-86}, we give the following corollary where the boundary is assumend to be convex.
\begin{cor}
Let $n\in\NN$, $n\geq 2$, $T,R>0$. Suppose $M$ is a compact Riemannian manifold with boundary of dimension $n$, such that the second fundamental form of $\partial M$ is non-negative. Assume $$\mu_T(\rho_-)<\left[38+80n^2\right]^{-1}.$$ Then, we have for any positive solution of the Neumann heat equation on $M$:
\begin{align}\label{eq:NeumannLiYau}
\frac 12 J \frac{\vert\nabla u\vert^2}{u^2}-\frac {\partial_t u}u\leq C_1+\frac{C_2}{J t}, \quad t\in (0,T],
\end{align}
where 
\[
\euler^{-16\kappa_T(\rho_-)}\leq J\leq 1
\]
and 
\[
C_1=4n^2\sqrt{\frac {4(1+2n^2)}{4(1+2n^2)-1}},\quad
C_2=\frac 1{2n^2}\left(1-\frac{1}{4+8n^2}\right)
\]
\end{cor}
The main ingredient of the above theorem are bounds of the solution of the following problem. In \cite{Olive-19}, it was solved using that global volume doubling of an ambient space $N$, an upper bound for the heat kernel on $N$, and the induced doubling condition in $M\subset N$ imply an a priori Neumann heat kernel bound. We do not need anything of this and use the Kato condition instead.
\begin{lem}\label{lem:harnackneumann}
Assume $c>1$ and $$\mu_T(\rho_-)<1/2(c-1).$$
There exists a unique solution for
\begin{align}\label{eq:harnackneumann}
\begin{cases}
\Delta J-\partial_t J -c \frac{\vert \nabla J\vert^2}J-2J \rho_-=0& \text{ in } M\times (0,T],\\
\partial_\nu J=0 & \text{ on } \partial M\times (0,T],\\
J=1 & \text{ on } M\times\{0\},
\end{cases}
\end{align}
satisfying
\[
\euler^{-16\mu_T(\rho)} \leq J\leq 1.
\]
\end{lem}
\begin{proof}
Using the transformation $w=J^{-(c-1)}$, the above problem transforms into 
\begin{align}
\begin{cases}
\partial_t w= \Delta w-Vw,& \text{ in } M\times(0,T],\\
 \partial_\nu w=0 &\text{ on }\partial M\times(0,T]\\
w=1 &\text{ on } M\times \{0\},
\end{cases}
\end{align}
where $V=2(c-1)\rho_-$. Due to our assumption $\mu_T(V)<1$. Since $-\Delta$ generates a Dirichlet form in $L^2(M)$, the function $w=\euler^{-t(-\Delta -V)}1$ solves this problem. Since $M$ is stochastically complete, Trotter-Kato gives $w\geq 1$, implying $J\leq 1$. Furthermore, $\Vert P_t\Vert_{\infty,\infty}\leq 1$.  Hence, by [Voigt86],
$$\Vert w\Vert_\infty\leq \Vert\euler^{-t(-\Delta -V)}\Vert_{\infty,\infty}\leq\left(\frac{1}{1-\mu_T(V)}\right)^{1+\frac tT}\leq \euler^{16(c-1)\mu_T(\rho_-)},$$
giving the lower bound on $J$.
\end{proof}
\begin{proof}[Proof of Theorem~\ref{gradientestimate}]

To obtain our result, we need to follow the proof of \cite[Theorem~1.1]{Olive-19}. We repeat the necessary steps and explain were the Kato condition enters. Denote by $\psi\colon [0,\infty)\to[0,\infty)$ a $C^2$-function such that 
\[
\begin{cases}\psi(r)\leq H\colon & r\in [0,1/2)\\ \psi(r)=H\colon r\in[1,\infty)\end{cases},
\]
$\psi(0)=0$, $0\leq \psi'\leq 2H$, $\psi'(0)=H$, and $\psi''\geq -H$. Let $r\colon M\to[0,\infty)$ the distance function to the boundary, $\psi(x):=\psi\left(r(x)/R\right)$, $\varphi:=(1+\varphi)^2$, and $\tilde\varphi:=\alpha\varphi$ for $\alpha>0$. Then, we have 
\begin{align}\label{boundaryfunction}
\alpha\leq \tilde\varphi\leq \alpha(1+H)^2, \vert\nabla\tilde\varphi\vert\leq \frac 4R\alpha H(1+H), \Delta\tilde\varphi\geq -2\alpha(1+H)\left(\frac H{R^2}+\frac 2R(n-1)H(3H+1)\right).
\end{align}
Let $u$ be a positive solution of the Neumann heat equation and set $f:=\ln u$, $\beta,\epsilon>0$, $J$ the solution of \eqref{eq:harnackneumann}, $c=(3+\alpha^{-1})\beta^{-1}$, and
\[
G(x,t):= t\left(\tilde\varphi J (\vert\nabla f\vert^2 +\epsilon)-\partial_t f\right).
\]
Let $(p,t_0)$ is a maximum of $G$ in $M\times (0,T]$, for $T>0$. Then w.l.o.g. $t_0>0$. Assuming $p\in\partial M$ leads to a contradiction due to the interior $R$-rolling ball condition and second fundamental form bounded below, cf.~\cite{Wang-97}. Thus, $p\in M\setminus\partial M$ and $(p,t_0)$ is a local maximum. Hence, $\nabla G=0$, $\partial_tG\geq 0$, $\Delta G\leq 0$, and $\Delta G-\partial_t G\leq 0$ in $(p,t_0)$. W.l.o.g., we can assume $G(p,t_0)>0$. Putting $Q=G/t$ and $g=\vert \nabla f\vert^2+\epsilon$, a lengthy calculation using Cauchy-Schwarz and the binomial formula several times gives
\begin{align}
&(\Delta-\partial_t)Q+2\nabla f\nabla Q\geq Jg\Delta\tilde\varphi+2J\nabla\tilde\varphi\nabla(\vert\nabla f\vert^2)+2gJ\nabla\tilde\varphi\nabla f-\beta\vert\nabla\tilde\varphi\vert^2gJ\\
&+2(1-\beta)\tilde\varphi J\vert\partial_i\partial_j f\vert^2-\beta Jg\tilde\varphi\vert\nabla f\vert^2+\tilde\varphi g\left(\Delta J-\partial_t J-\left(3+\frac 1\alpha\right)\frac 1\beta\frac{\vert\nabla J\vert^2}J-2J\rho_-\right).
\end{align}

Now we use Lemma~\ref{lem:harnackneumann} with our choice for $c$, which is the step that is solved in \cite{Olive-19} by the a priori bound of the Neumann heat kernel. This gives
\begin{align}
&(\Delta-\partial_t)Q+2\nabla f\nabla Q\geq Jg\Delta\tilde\varphi+2J\nabla\tilde\varphi\nabla(\vert\nabla f\vert^2)+2gJ\nabla\tilde\varphi\nabla f-\beta\vert\nabla\tilde\varphi\vert^2gJ\\
&+2(1-\beta)\tilde\varphi J\vert\partial_i\partial_j f\vert^2-\beta Jg\tilde\varphi\vert\nabla f\vert^2.
\end{align}
The relation of $G$ and $Q$, the fact that $(p,t_0)$ is a local maximum, the binomial formula, and 
\[
\sum_{i,j=1}^n\vert\partial_i\partial_j f\vert^2\geq \frac 1{n^2}\left(\vert\nabla f\vert^2-\partial_t f\right)^2
\]
 imply
\begin{align}
Q
&\geq t_0J\left(g\Delta\tilde\varphi +2g\nabla\tilde\varphi\nabla f-\beta\vert\nabla\tilde\varphi\vert^2g\right.\\
&\quad \left.+\frac {2(1-\beta)\tilde\varphi-\alpha}{n^2}\left(\vert \nabla f\vert^2-\partial_t f\right)^2-\frac 4\alpha\vert\nabla\tilde\varphi\vert^2\vert\nabla f\vert^2-2\beta g\tilde\varphi\vert\nabla f\vert^2\right).
\end{align}
This implies
\begin{align}
0\geq t_0 J\left(\vert\nabla f\vert^2(\Delta\tilde\varphi-(\beta+\frac 4\alpha)\vert\nabla\tilde\varphi\vert^2+O(\epsilon))-2\vert\nabla\tilde\varphi\vert\vert\nabla f\vert^3+\frac{2(1-\beta)\tilde\varphi-\alpha}{n^2}Q^2\right.\\
\left.+\vert\nabla f\vert^4\left(\left(1-\tilde\varphi J\right)^2\frac{2(1-\beta)\tilde\varphi-\alpha}{n^2}-\beta\tilde\varphi\right)\right)-(1+O(\epsilon))Q+O(\epsilon)
\end{align}
Choosing 
\begin{align*}
\alpha&=\frac 1{2(1+H)^2}, \quad \beta=(4+8n^2(1+H)^2)^{-1}, \quad A=(16n^2(1+H)^2)^{-1}, \\
C&=\frac{2H}{R(1+H)}\left((4+8n^2(1+H)^2)^{-1}+8(1+H)^2\right),\quad B=\frac{4H}{R(1+H)}, \\
E&=\frac 1{2n^2(1+H)^2}\left(1-\frac 1{4+8n^2(1+H)^2}\right),
\end{align*}
we get 
\begin{align*}
\left(1-\tilde\varphi J\right)^2\frac{2(1-\beta)\tilde\varphi-\alpha}{n^2}-\beta\tilde\varphi&\geq A,\quad
\Delta\tilde\varphi-(\beta+\frac 4\alpha\geq -C+O(\epsilon)=:-C_\epsilon,\\
-2\vert\nabla \tilde\varphi\vert&\geq -B,\quad
\frac{2(1-\beta)\tilde\varphi-\alpha}{n^2}\geq E.
\end{align*}
Hence,
\begin{align}
0\geq t_0 J\left(-C_\epsilon\vert\nabla f\vert^2-B\vert\nabla f\vert^3+A\vert\nabla f\vert^4+EQ^2\right)-(1+O(\epsilon))Q+O(\epsilon).
\end{align}
Denoting $y=\vert\nabla f\vert^2$ and 
\[
\tilde D=8n^2(1+H)^2\left(\frac{128 n^2H^2}{R^2}+\frac{2H}{4R(1+H)(1+2n^2(1+H)^2}+\frac{16H(1+H)}{R}\right),
\]
we have
\[
Ay^2-By^{3/2}-C_\epsilon y\geq -\tilde D+O(\epsilon).
\]
Thus,
\[
0\geq EJG^2-(1+O(\epsilon))G-\tilde Dt_0^2 J+O(\epsilon).
\]
Solving for $G$ and keeping in mind that $G$ is positive in $(p,t_0)$ leads to the claim.

\end{proof}
From Theorem~\ref{gradientestimate}, we infer by standard techniques the following.
\begin{prop}\label{prop:harnack}
Let $n\in\NN$, $n\geq 2$, $H,T,R>0$. Suppose $M$ is a compact Riemannian manifold with boundary of dimension $n$, such that the second fundamental form of $\partial M$ is bounded below by $-H$, and that satisfies the interior $R$-rolling ball condition for $R$ small enough as in Remark~\ref{rmk:Rsmall}. Assume $$\mu_T(\rho_-)<\left[2\left((3+2(1+H)^2)(4+8n^2(1+H)^2)-1\right)\right]^{-1}.$$ Then, we have for any positive solution $u$ of the Neumann heat equation on $M$ and $0<t_1\leq t_2\leq T$, $x,y\in M$,
\begin{align}
u(t_1,x)\leq \left(\frac{t_2}{t_1}\right)^{C_1T+C_2}u(t_2,x)
\end{align}
and
\begin{align}\label{prop:harnackeq}
u(x,t_1)\leq u(y,t_2)\left(\frac{t_2}{t_1}\right)^{C_2\euler^{16\mu_T(\rho_-)}}\exp\left(C_1(t_2-t_1)+\frac{\euler^{16\mu_T(\rho_-)}(1+H)^2d(x,y)^2}{2(t_2-t_1)}\right).
\end{align}
\end{prop}
From this, it is easy to get the following Neumann heat kernel bound.
\begin{thm}
Let $n\in\NN$, $D,T,R>0$. Suppose $M$ is a compact Riemannian manifold with boundary of dimension $n$, such that the second fundamental form of $\partial M$ is bounded below by $-H$, and that satisfies the interior $R$-rolling ball condition for $R$ small enough as in Remark~\ref{rmk:Rsmall}. Assume $$\mu_{2D^2}(\rho_-)<\left[2\left((3+2(1+H)^2)(4+8n^2(1+H)^2)-1\right)\right]^{-1}.$$ Then, the Neumann heat kernel satisfies for all $x\in M$ and $t\in(0,D^2]$
\begin{align}\label{thm:Nhk}
h_t(x,x)\leq \frac{C_{T,\mu,H,R}}{\Vol(B(x,\sqrt t))}.
\end{align}

where $$C_{T,\mu,H,R}=2^{C_1\mu_T(\rho_-)}\euler^{C_1T+16\mu_T(\rho_-)}(1+H)^2.$$
\end{thm}
\begin{proof}
According to Proposition~\ref{prop:harnack} and since $h_t(x,x)$ is non-increasing in $t$, we have for all $t\in(0,D^2]$ and $y\in M$ such that $d(x,y)\leq \sqrt t$
\[
h_t(x,x)\leq h_{2t}(x,y)C_{T,\mu,H,R}\leq h_{3t}(x,x)C_{T,\mu,H,R}^2\leq C_{T,\mu,H,R}^2h_t(x,x).
\]
Integrating the first inequality over $B(x,\sqrt t)$ and noting that $\int h_t(x,y)\dvol y=1$, we get the result.
\end{proof}
\subsection{Proof of Theorem~\ref{thm:neumanneigen}}% and \ref{thm:main4}}

\begin{proof}[Proof of Theorem~\ref{thm:neumanneigen}]
The first steps of the proof are the same as in the proof of \cite[Theorem~1]{ChenLi-97}. The proof differs after \eqref{eigenheatest}.

The function
 \begin{equation}
  F(x,t):= \int_\Omega h(x,y,t)f(y)dy
 \end{equation}
 solves the heat equation on $M$ with Neumann boundary conditions on $\partial M$ and initial condition $F(x,0) = f(x)$. Consider the function
 \begin{equation}
  g(x,t):= \int_\Omega h(x,y,t)\left( f(y)-F(x,t) \right)^2 dy.
 \end{equation}
 By definition, we have
 \begin{align}\label{eq3.117}
   \int_M g(x,t)dx = & \int_M\int_M h(x,y,t)f^2(y)dydx-\int_M F^2(x,t)dx  \nonumber\\
   &= \int_M f^2(y) dy - \int_M F^2(x,t)dx  \nonumber \\
   &=-\int_0^t\frac{\partial}{\partial s}\left(\int_M F^2(x,s)dx\right) ds \nonumber\\
   &=-2\int_0^t\int_M F(x,s)\Delta F(x,s)dx ds \nonumber \\
   &=2\int_0^t\int_M |\nabla F|^2(x,s)dx ds.
 \end{align}
On the other hand, we have
\begin{align*}
 \frac{\partial}{\partial t} \int_M |\nabla F|^2(x,t)dx &= 2\int_M \nabla F \cdot \nabla F_t (x,t)dx\\
 &= -2\int_M \Delta F F_t (x,t) dx\\
 &= -2\int_M F_t^2(x,t)dx\leq 0,
\end{align*}
hence we conclude that for any $t>0$
\begin{equation}
 \int_M |\nabla F|^2(x,t)dx \leq \int_M |\nabla f|^2(x)dx.
\end{equation}
Thus, from \eqref{eq3.117} we have the estimate
\begin{equation}\label{eq3.120}
 \int_M g(x,t)dx\leq 2t\int_M |\nabla f|^2(x)dx.
\end{equation}

At this point, the proof differs from the one in \cite{ChenLi-97}.
We have
\begin{align}\label{eq3.121}
 \int_M g(x,t)dx 
 &=\int_M \int_M h(x,y,t)\left( f(y)-F(x,t)\right)^2dydx  \nonumber \\
&\geq \int_M \inf_{y\in M} h(x,y,t) \int_M \left( f(y)-F(x,t) \right)^2dydx \nonumber \\
&\geq \inf_{a\in \RR}\int_M \left(f(y)-a \right)^2dy \int_M\inf_{y\in M} h(x,y,t)dx.
 \end{align}
Using the variational principle, we have from \eqref{eq3.120} and \eqref{eq3.121} that
\begin{equation}
 \eta_1 \geq \frac{1}{2t} \int_{M} \inf_{y\in M} h(x,y,t)dx. \label{eigenheatest}
\end{equation}

Note that, since $M$ is compact and $h$ is smooth, the infimum is attained at some point in $M$. Since $M$ is stochastically complete, 
for a fixed $\tilde{t}>0$, for all $x\in M$ there exists $y^*\in M$ such that $h(x,y^*,\tilde{t})\geq \frac{1}{2\Vol(M)}$ (since otherwise $\int_M h(x,y,\tilde{t})dx \leq \frac{1}{2}$).

Fix $x\in M$. For $y,z\in M$, we can use \eqref{prop:harnackeq} on $h(x,y,t)$ for $t_1 = \frac{t}{2}$ and $t_2 =t$ $t\leq T$, where $t>0$ will be determined later, to obtain
\begin{equation}\label{eq3.124}
 h_{t/2}(x,y)\leq h_t(x,z)2^{C_2\euler^{16\mu_T(\rho_-)}}\exp\left(C_1t/2+\euler^{16\mu_T(\rho_-)}(1+H)^2D^2\right).
\end{equation}
Let $y^*\in M$ be such that $h_{t/2}(x,y^*) \geq \frac{1}{2\Vol(M)}$. Then choosing $y=y^*$ and letting $z$ be the point where $h$ is minimum in \eqref{eq3.124}, we have that
\begin{equation*}
\begin{split}
  \inf_{y\in \Omega} h_t(x,y) &= h_t(x,z)\\
 &\geq h_{t/2}(x,y^*)2^{-C_2\euler^{16\mu_T(\rho_-)}}\exp\left(-C_1t/2-\euler^{16\mu_T(\rho_-)}(1+H)^2D^2\right)\\
 &\geq \frac{1}{2\Vol (M)}2^{-C_2\euler^{16\mu_T(\rho_-)}}\exp\left(-C_1t/2-\euler^{16\mu_T(\rho_-)}(1+H)^2D^2\right).
 \end{split}
\end{equation*}
Thus, since the choice of $x\in M$ was arbitrary, using \eqref{eigenheatest}, we get
\begin{equation}
 \eta_1 \geq 
\frac 1{4t} 2^{-C_2\euler^{16\mu_T(\rho_-)}}\exp\left(-C_1t/2-\euler^{16\mu_T(\rho_-)}(1+H)^2D^2\right).
\end{equation}
Choosing $t=\min(T,D^2/2)$ yields the claim.
\end{proof}

\bibliographystyle{alpha}
%\bibliography{EigenvaluesKato.bib}

\begin{thebibliography}{ROSWZ19}

\bibitem[AC13]{Andrews-Clutterbuck-13}
Ben Andrews and Julie Clutterbuck.
\newblock Sharp modulus of continuity for parabolic equations on manifolds and
  lower bounds for the first eigenvalue.
\newblock {\em Anal. PDE}, 6(5):1013--1024, 2013.

\bibitem[Aub07]{Aubry-07}
E.~Aubry.
\newblock Finiteness of {$\pi_1$} and geometric inequalities in almost positive
  {R}icci curvature.
\newblock {\em Ann. Sci. \'Ecole Norm. Sup. (4)}, 40(4):675--695, 2007.

\bibitem[Bus82]{Buser-82}
Peter Buser.
\newblock A note on the isoperimetric constant.
\newblock {\em Ann. Sci. \'{E}cole Norm. Sup. (4)}, 15(2):213--230, 1982.

\bibitem[Car16]{Carron-16}
Gilles Carron.
\newblock Geometric inequalities for manifolds with {R}icci curvature in the
  {K}ato class.
\newblock 2016.
\newblock https://arxiv.org/abs/1612.03027 [math.DG].

\bibitem[Che70]{Cheeger-70}
Jeff Cheeger.
\newblock A lower bound for the smallest eigenvalue of the {L}aplacian.
\newblock pages 195--199, 1970.

\bibitem[Che75]{Cheng-75}
Shiu-Yuen Cheng.
\newblock Eigenvalue comparison theorems and geometric applications.
\newblock {\em Math. Z.}, 143(3):289--297, 1975.

\bibitem[Che90]{Chen-90}
Roger Chen.
\newblock Neumann eigenvalue estimate on a compact {R}iemannian manifold.
\newblock {\em Proc. Amer. Math. Soc.}, 108(4):961--970, 4 1990.

\bibitem[CL97]{ChenLi-97}
Roger Chen and Peter Li.
\newblock On {P}oincar{\'e} {T}ype {I}nequalities.
\newblock {\em Trans. Amer. Math. Soc.}, 349(4):1561--1585, 1997.

\bibitem[CR18]{CarronRose-18}
Gilles Carron and Christian Rose.
\newblock Geometric and spectral estimates based on spectral {R}icci curvature
  assumptions.
\newblock 2018.
\newblock Preprint. https://arxiv.org/abs/1808.06965 [math.DG].

\bibitem[CW97]{ChenWang-97}
Mufa Chen and Fengyu Wang.
\newblock General formula for lower bound of the first eigenvalue on
  {R}iemannian manifolds.
\newblock {\em Sci. China Ser. A}, 40(4):384--394, 1997.

\bibitem[Gal88]{Gallot-88}
S.~Gallot.
\newblock Isoperimetric inequalities based on integral norms of {R}icci
  curvature.
\newblock {\em Ast\'erisque}, (157-158):191--216, 1988.
\newblock Colloque Paul L{\'e}vy sur les Processus Stochastiques (Palaiseau,
  1987).

\bibitem[Gro80]{Gromov-80}
Mikhael Gromov.
\newblock Paul levy's isoperimetric inequality.
\newblock 1980.

\bibitem[G{\"u}n14]{Batu-14}
Batu G{\"u}neysu.
\newblock Kato's inequality and form boundedness of {K}ato potentials on
  arbitrary {R}iemannian manifolds.
\newblock {\em Proc. Amer. Math. Soc.}, 142(4):1289--1300, 2014.

\bibitem[HSC01]{HebischSaloffCoste-01}
Waldemar Hebisch and Laurent Saloff-Coste.
\newblock On the relation between elliptic and parabolic {H}arnack
  inequalities.
\newblock {\em Ann. Inst. Fourier}, 51(5):1437--1481, 2001.

\bibitem[HW07]{HangWang-07}
Fengbo Hang and Xiaodong Wang.
\newblock A remark on {Z}hong-{Y}ang's eigenvalue estimate.
\newblock {\em Int. Math. Res. Not. IMRN}, (18):Art. ID rnm064, 9, 2007.

\bibitem[Kr{\"o}92]{Kroger-92}
Pawel Kr{\"o}ger.
\newblock On the spectral gap for compact manifolds.
\newblock {\em J. Differential Geom.}, 36(2):315--330, 1992.

\bibitem[Lic58]{Lichnerowicz-58}
Andr\'{e} Lichnerowicz.
\newblock G\'{e}om\'{e}trie des groupes de transformations.
\newblock pages ix+193, 1958.

\bibitem[LY80]{LiYau-79}
Peter Li and Shing~Tung Yau.
\newblock Estimates of eigenvalues of a compact {R}iemannian manifold.
\newblock pages 205--239, 1980.

\bibitem[LY86]{LiYau-86}
Peter Li and Shing-Tung Yau.
\newblock On the parabolic kernel of the {S}chr\"{o}dinger operator.
\newblock {\em Acta Math.}, 156(3-4):153--201, 1986.

\bibitem[MR19]{MuenchRose-19}
Florentin M{\"u}nch and Christian Rose.
\newblock Spectrally positive {B}akry-{\'e}mery {R}icci curvature on graphs.
\newblock 2019.
\newblock arXiv:1912.06438 [math.DG].

\bibitem[PS98]{PetersenSprouse-98}
Peter Petersen and Chadwick Sprouse.
\newblock Integral curvature bounds, distance estimates and applications.
\newblock {\em J. Differential Geom.}, 50(2):269--298, 1998.

\bibitem[PW97]{PetersenWei-97}
P.~Petersen and G.~Wei.
\newblock Relative volume comparison with integral curvature bounds.
\newblock {\em Geom. Funct. Anal.}, 7(6):1031--1045, 1997.

\bibitem[RO19]{Olive-19}
Xavier Ramos~Oliv{\'e}.
\newblock Neumann {L}i-{Y}au gradient estimate under integral {R}icci curvature
  bounds.
\newblock {\em Proc. Amer. Math. Soc.}, 147:411--426, 2019.

\bibitem[Ros19a]{Rose-19}
Christian Rose.
\newblock Almost positive {R}icci curvature in {K}ato sense - an extension of
  {M}yers' theorem.
\newblock 2019.
\newblock arXiv:1907.07440 [math.DG].

\bibitem[Ros19b]{Rose-16a}
Christian Rose.
\newblock {L}i-{Y}au gradient estimate for compact manifolds with negative part
  of {R}icci curvature in the {K}ato class.
\newblock {\em Ann. Glob. Anal. Geom.}, 55(3):443--449, Apr 2019.

\bibitem[ROSWZ19]{OliveSetoWeiZhang-19}
Xavier Ramos~Oliv{\'e}, Shoo Seto, Guofang Wei, and Qi~S. Zhang.
\newblock Zhong–{Y}ang type eigenvalue estimate with integral curvature
  condition.
\newblock {\em Math. Z.}, 2019.

\bibitem[RS17]{RoseStollmann-15}
Christian Rose and Peter Stollmann.
\newblock The {K}ato class on compact manifolds with integral bounds of {R}icci
  curvature.
\newblock {\em Proc. Amer. Math. Soc.}, 145:2199--2210, 2017.

\bibitem[RS18]{RoseStollmann-18}
Christian Rose and Peter Stollmann.
\newblock Manifolds with {R}icci curvature in the {K}ato class: heat kernel
  bounds and applications.
\newblock 2018.
\newblock to appear in the proceedings for the conference \emph{{A}nalysis and
  geometry on graphs and manifolds}, arXiv:1608.04221 [math.DG].

\bibitem[Sim82]{Simon-82}
Barry Simon.
\newblock Schr\"odinger semigroups.
\newblock {\em Bull. Amer. Math. Soc. (N.S.)}, 7(3):447--526, 1982.

\bibitem[SV96]{StollmannVoigt-96}
Peter Stollmann and J{\"u}rgen Voigt.
\newblock Perturbation of {D}irichlet forms by measures.
\newblock {\em Potential Anal.}, 5(2):109--138, 1996.

\bibitem[SZ07]{ShiZhang-07}
Yu~Min Shi and Hui~Chun Zhang.
\newblock Lower bounds for the first eigenvalue on compact manifolds.
\newblock {\em Chinese Ann. Math. Ser. A}, 28(6):863--866, 2007.

\bibitem[TZ15]{TianZhang-15}
Gang Tian and Qi~S. Zhang.
\newblock A compactness result for {F}ano manifolds and {K}{\"a}hler {R}icci
  flows.
\newblock {\em Math. Ann.}, 362(3):965--999, Aug 2015.

\bibitem[TZ16]{TianZhang-16}
Gang Tian and Zhenlei Zhang.
\newblock Regularity of {K}{\"a}hler–{R}icci flows on {F}ano manifolds.
\newblock {\em Acta Math.}, 216(1):127--176, 2016.

\bibitem[Voi86]{Voigt-86}
J{\"u}rgen Voigt.
\newblock Absorption semigroups, their generators, and {S}chr\"odinger
  semigroups.
\newblock {\em J. Funct. Anal.}, 67(2):167--205, 1986.

\bibitem[Wan97]{Wang-97}
Jiaping Wang.
\newblock Global heat kernel estimates.
\newblock {\em Pacific J. Math.}, 178(2):377--398, 1997.

\bibitem[Yan90]{Yang-90}
Hong~Cang Yang.
\newblock Estimates of the first eigenvalue for a compact {R}iemann manifold.
\newblock {\em Sci. China Ser. A}, 33(1):39--51, 1990.

\bibitem[ZW17]{ZhangWang-17}
Yuntao Zhang and Kui Wang.
\newblock An alternative proof of lower bounds for the first eigenvalue on
  manifolds.
\newblock {\em Math. Nachr.}, 290(16):2708--2713, 2017.

\bibitem[ZY84]{ZhongYang-84}
Jia~Qing Zhong and Hong~Cang Yang.
\newblock On the estimate of the first eigenvalue of a compact {R}iemannian
  manifold.
\newblock {\em Sci. Sinica Ser. A}, 27(12):1265--1273, 1984.

\bibitem[ZZ18]{ZhangZhu-18}
Qi~S. Zhang and Meng Zhu.
\newblock Li-{Y}au gradient bounds on compact manifolds under nearly optimal
  curvature conditions.
\newblock {\em J. Funct. Anal.}, 275(2):478--515, 2018.

\end{thebibliography}

\end{document}